\documentclass[12pt]{article}

\usepackage[T1]{fontenc}
\usepackage[utf8]{inputenc}
\usepackage{lmodern}
\usepackage{amsmath,amsthm,amsfonts,amssymb}
\usepackage[a4paper,margin=27mm]{geometry}
\usepackage{microtype}
\usepackage[hidelinks]{hyperref}

\numberwithin{equation}{section}

\newtheorem{theorem}{Theorem}[section]
\newtheorem{lemma}[theorem]{Lemma}
\newtheorem{proposition}[theorem]{Proposition}
\newtheorem{corollary}[theorem]{Corollary}

\theoremstyle{definition}

\newtheorem{example}[theorem]{Example}
\theoremstyle{remark}
\newtheorem{remark}[theorem]{Remark}

\newcommand{\Z}{\mathbb Z}
\newcommand{\one}[1]{\mathbf 1_{\{#1\}}}
\newcommand{\Ccal}{\mathcal C}
\newcommand{\Rcal}{\mathcal R}
\newcommand{\Zcal}{\mathcal Z}
\newcommand{\Ecal}{\mathcal E}
\newcommand{\Bcal}{\mathcal B}

\hypersetup{
  pdftitle={
    Restricted Coprime Representations Modulo 6P:
    Exact CRT Formulas and Positivity Thresholds
  },
  pdfauthor={Andres M. Salazar},
  pdfkeywords={
    additive representations,
    coprimality,
    Chinese remainder theorem,
    canonical residues,
    finite local densities,
    Nagell totient,
    quasipolynomials,
    positivity thresholds
  }
}

\title{
Restricted Coprime Representations Modulo
\texorpdfstring{$6P$}{6P}:\\
Exact CRT Formulas and Positivity Thresholds
}

\author{
Andr\'es M. Salazar\thanks{
\texttt{andresmsalazar@javerianacali.edu.co}
}\\
Department of Natural Sciences and Mathematics\\
Pontificia Universidad Javeriana Cali, Colombia
}

\date{}

\begin{document}

\maketitle

\begin{abstract}
Let $5\leq p_1<\cdots<p_r$ be primes, put
$P=p_1\cdots p_r$, and let $g_P(2n)$ count the unordered
decompositions $2n=h+k$ for which both summands are coprime to
$6P$.  We derive an explicit finite inclusion--exclusion formula
for $g_P(2n)$ by parametrizing the admissible residue classes and
combining the local restrictions through the Chinese remainder
theorem.  The associated complete-modulus correlation is a
specialization of Nagell's totient function and determines both
the finite local density and the affine increment
\[
g_P\bigl(2(n+6P)\bigr)-g_P(2n)=\kappa_P(n).
\]

A parameter-density decomposition gives a residue-sensitive
positivity criterion and the uniform sufficient condition
\[
n\geq
3(3^r-1)\prod_{i=1}^{r}\frac{p_i}{p_i-2}+2
\quad\Longrightarrow\quad
g_P(2n)>0.
\]
Maximizing this bound over squarefree moduli whose prime factors
do not exceed a prime $p$ gives an explicit threshold
$\mathcal N(p)$ satisfying
$\log\mathcal N(p)\sim(\log3)p/\log p$.  We also describe the
complete prime cutoff at $\sqrt{2n}$ and its exact additive-prime
interpretation. We also evaluate $g_P(2tP)$ explicitly using the
real character of period $3$. The positivity thresholds are elementary
and fully explicit; lower-bound sieves provide smaller asymptotic
scales, and the present thresholds do not establish positivity at
the complete cutoff for all $n\geq13$.
\end{abstract}

\medskip


\medskip


\section{Introduction}

Representations of integers as sums of arithmetically restricted
summands are a classical subject of additive number theory
\cite{Nathanson1996}. Related counting problems for additive
compositions and partitions under coprimality constraints have been
studied, for example, in
\cite{BubboloniLucaSpiga2012,BubboloniLuca2022}. The restriction
considered here is different: coprimality is imposed against a fixed
squarefree modulus rather than between the summands themselves, and
the objective is an exact finite-modulus count for unordered two-term
representations.

Given
\[
P=p_1p_2\cdots p_r,
\]
where $5\leq p_1<\cdots<p_r$ are primes, we count
representations of $2n$ in which neither summand is divisible
by $2$, by $3$, or by any prime dividing $P$.

The one-prime case was studied in \cite{Salazar2026}. The present
paper extends that calculation from $6p$ to an arbitrary
squarefree modulus $6P$.  The exact formula is a finite
Legendre-type inclusion--exclusion expansion with explicit CRT
representatives.  When $r=1$, it specializes to an equivalent
form of the one-prime formula.

We use the canonical remainder
$\delta_q(x)=x-q\lfloor x/q\rfloor$ and the collision indicator
$\varepsilon_i(n)=\mathbf1_{\{\delta_{p_i}(n)=0\}}$.
These conventions agree with the notation of \cite{Salazar2026}.

The argument yields four structural consequences.  First, a
$6$-wheel parametrization reduces the representation problem to
one integer parameter.  Each prime $p_i$ removes at most two
residue classes, and those classes merge exactly when $p_i$ divides
$n$.  Second, the resulting inclusion--exclusion formula has
\[
\prod_{i=1}^{r}\bigl(3-\varepsilon_i(n)\bigr)
\]
nonzero coefficients, and hence at most $3^r$ nonzero summands.  For fixed $P$ its
size is independent of $n$.

Third, the complete-modulus correlation is the Nagell totient
$\theta(2n,6P)$ \cite{Cohen1960}.  In the present
squarefree setting it has the explicit value
\[
\kappa_P(n)=
\left(1+\one{\delta_3(n)=0}\right)
\prod_{i=1}^{r}
\left(p_i-2+\one{\delta_{p_i}(n)=0}\right).
\]
This same quantity is the affine increment of $g_P(2n)$ and the
finite local-density factor.

Fourth, the parameter-density decomposition produces a single
positivity mechanism from which several useful thresholds follow.
In particular,
\[
n\geq
3(3^r-1)
\prod_{i=1}^{r}\frac{p_i}{p_i-2}+2
\]
is sufficient for $g_P(2n)>0$.  Maximizing this condition over
all squarefree products supported on primes at most $p$ gives a
largest-prime threshold $\mathcal N(p)$ with
\[
\log\mathcal N(p)
\sim
(\log3)\frac{p}{\log p}.
\]
The complete cutoff obtained by including all primes up to
$\sqrt{2n}$ is then treated separately and characterized exactly
in terms of the surviving additive representations.
We also obtain an exact evaluation on $n=tP$, with a correction
expressed through the real character of period $3$.
The complete-modulus product and the sieve mechanism have classical
antecedents. The emphasis here is on explicit canonical CRT counts
in the truncated interval, their unordered interpretation, and
fully explicit sufficient thresholds. The latter are not claimed
to be optimal; a comparison with lower-bound sieves is given below.

\section{Canonical residues and the
\texorpdfstring{$6$}{6}-wheel parametrization}

Throughout the paper, let
\[
P=\prod_{i=1}^{r}p_i,
\qquad
M=6P,
\]
where $r\geq1$ and $5\leq p_1<\cdots<p_r$ are primes.

For every positive integer $n$, define
\begin{equation}
\label{eq:g-definition}
g_P(2n)
=
\#\left\{
(h,k)\in\Z_{>0}^{\,2}
\mathrel{}\middle|\mathrel{}
h+k=2n,\ 
h\leq k,\ 
\gcd(h,M)=\gcd(k,M)=1
\right\}.
\end{equation}

Selecting the smaller summand gives the equivalent expression
\begin{equation}
\label{eq:g-short-interval}
g_P(2n)
=
\#\left\{
h\in\Z
\mathrel{}\middle|\mathrel{}
1\leq h\leq n,\ 
\gcd(h,M)=\gcd(2n-h,M)=1
\right\}.
\end{equation}
This form will be used later in the periodic interpretation and the evaluation on $n=tP$.

\subsection{Canonical residues and local parameters}

For $q\in\Z_{>0}$ and $x\in\Z$, define the canonical remainder
operator by
\begin{equation}
\label{eq:canonical-delta}
\delta_q(x)
=
x-q\left\lfloor\frac{x}{q}\right\rfloor.
\end{equation}
Thus
\[
0\leq\delta_q(x)\leq q-1.
\]
Moreover, for all $x,y\in\Z$,
\begin{equation}
\label{eq:delta-divisibility}
\delta_q(x)=\delta_q(y)
\quad\Longleftrightarrow\quad
q\mid x-y.
\end{equation}
In particular,
\[
\delta_q(x)=0
\quad\Longleftrightarrow\quad
q\mid x.
\]

The definition includes $q=1$, for which $\delta_1(x)=0$. This
convention will be useful when the Chinese-remainder modulus is
an empty product.

For every prime $p\geq5$, define
\begin{equation}
\label{eq:ab-definition}
a(p)
=
\frac{p\delta_6(p)-1}{6},
\qquad
b(p)
=
p-1-a(p).
\end{equation}

Since a prime $p\geq5$ satisfies
$\delta_6(p)\in\{1,5\}$, the parameters are explicitly given by
\[
\begin{array}{c|c|c|c}
\delta_6(p) & p & a(p) & b(p)\\
\hline
1 & 6c+1 & c & 5c\\
5 & 6c+5 & 5c+4 & c
\end{array}
\]
for a suitable nonnegative integer $c$. In particular,
$a(p),b(p)\in\{0,\dots,p-1\}$, and
\begin{align}
a(p)+b(p)
&=p-1,
\label{eq:ab-identities}
\\
6a(p)+1
&=p\delta_6(p),
\label{eq:a-identity}
\\
6b(p)+5
&=p\bigl(6-\delta_6(p)\bigr).
\label{eq:b-identity}
\end{align}

The parameters $a(p)$ and $b(p)$ identify the unique canonical
residues of $z$ for which $6z+1$ and $6z+5$, respectively, are
divisible by $p$.

\begin{lemma}[Local divisibility residues]
\label{lem:local-zero}
Let $p\geq5$ be prime. For every $z\in\Z$,
\begin{equation}
\label{eq:a-local-zero}
\delta_p(6z+1)=0
\quad\Longleftrightarrow\quad
\delta_p(z)=a(p),
\end{equation}
and
\begin{equation}
\label{eq:b-local-zero}
\delta_p(6z+5)=0
\quad\Longleftrightarrow\quad
\delta_p(z)=b(p).
\end{equation}
\end{lemma}

\begin{proof}
By \eqref{eq:a-identity},
\[
(6z+1)-(6a(p)+1)=6\bigl(z-a(p)\bigr).
\]
Since $p\nmid6$ and $p\mid6a(p)+1$, one has
$p\mid6z+1$ if and only if $p\mid z-a(p)$, which proves
\eqref{eq:a-local-zero}.  The second equivalence follows from
\eqref{eq:b-identity} in the same way.
\end{proof}

\subsection{The \texorpdfstring{$6$}{6}-wheel parametrization}

Put
\[
e(n)=\delta_3(n)
\]
and define
\begin{equation}
\label{eq:m-definition}
m(n)
=
\begin{cases}
\dfrac{n}{3}-1,
&
e(n)=0,
\\[6pt]
\dfrac{n-1}{3},
&
e(n)=1,
\\[6pt]
\dfrac{n-5}{3},
&
e(n)=2.
\end{cases}
\end{equation}
Each branch is integral. For every positive integer $n$, one has
$m(n)\geq-1$, and equality occurs only when $n=2$. Define
\begin{equation}
\label{eq:Z-definition}
\Zcal(n)
=
\left\{
z\in\Z
\mathrel{}\middle|\mathrel{}
0\leq z\leq m(n)
\right\}.
\end{equation}
This definition automatically gives
\[
\Zcal(2)=\varnothing,
\]
because $m(2)=-1$.

For $z\in\Zcal(n)$, define $H_z(n)$ and $K_z(n)$ by
\begin{equation}
\label{eq:parametrization-table}
\begin{array}{c|c|c}
e(n) & H_z(n) & K_z(n)\\
\hline
0
&
6z+1
&
6\bigl(m(n)-z\bigr)+5
\\[4pt]
1
&
6z+1
&
6\bigl(m(n)-z\bigr)+1
\\[4pt]
2
&
6z+5
&
6\bigl(m(n)-z\bigr)+5
\end{array}
\end{equation}

In every row,
\begin{equation}
\label{eq:parametrization-sum}
H_z(n)+K_z(n)=2n.
\end{equation}
The inequalities
\[
z\geq0
\qquad\text{and}\qquad
m(n)-z\geq0
\]
ensure that both summands are positive. Conversely, positivity
of the two summands forces $0\leq z\leq m(n)$.

\begin{proposition}[The $6$-wheel parametrization]
\label{prop:six-wheel}
Suppose that $\delta_3(n)=0$. Then
\[
z\longmapsto\{H_z(n),K_z(n)\}
\]
is a bijection from $\Zcal(n)$ onto the unordered
representations of $2n$ whose two summands are coprime to $6$.

Suppose that $\delta_3(n)\in\{1,2\}$. Then the same map is
surjective onto the unordered representations of $2n$ whose
summands are coprime to $6$, and its fibers are precisely the
orbits of the involution
\[
\tau_n(z)=m(n)-z.
\]
A fixed parameter exists exactly when $m(n)$ is even. When it
exists, it gives
\[
H_z(n)=K_z(n)=n.
\]
\end{proposition}

\begin{proof}
Every integer coprime to $6$ has canonical remainder $1$ or $5$
under $\delta_6$. Since
\[
\delta_6(2n)
=
\begin{cases}
0, & \delta_3(n)=0,\\
2, & \delta_3(n)=1,\\
4, & \delta_3(n)=2,
\end{cases}
\]
the only possible pairs of canonical residues of two summands
coprime to $6$ are
\[
(1,5)\ \text{or}\ (5,1),
\qquad
(1,1),
\qquad
(5,5),
\]
respectively.

If $\delta_3(n)=0$, exactly one summand has canonical remainder
$1$ under $\delta_6$. Writing that summand uniquely as $6z+1$
gives the first row of \eqref{eq:parametrization-table}. The
positivity conditions give precisely $z\in\Zcal(n)$. Hence every
unordered representation occurs exactly once.

If $\delta_3(n)=1$, both summands have canonical remainder $1$.
If $\delta_3(n)=2$, both have canonical remainder $5$. In either
case,
\[
H_{m(n)-z}(n)=K_z(n),
\qquad
K_{m(n)-z}(n)=H_z(n).
\]
Thus interchanging the summands corresponds exactly to replacing
$z$ by $m(n)-z$, and the fibers are the orbits of $\tau_n$.

A fixed parameter satisfies
\[
z=m(n)-z,
\]
or equivalently $2z=m(n)$. At such a parameter,
\eqref{eq:parametrization-sum} gives
\[
H_z(n)=K_z(n)=n.
\]

\end{proof}

\section{Local exclusions and the exact CRT formula}

Proposition~\ref{prop:six-wheel} reduces the original
representation problem to the admissibility of the parameters
$z\in\Zcal(n)$. For each prime divisor $p_i$ of $P$,
divisibility of either summand excludes at most two canonical
residue classes of $z$. We now identify these classes, determine
exactly when they coincide, and combine the resulting local
conditions by inclusion--exclusion and the Chinese remainder
theorem.

For brevity, write
\[
a_i=a(p_i),
\qquad
b_i=b(p_i).
\]

For every $i\in\{1,\dots,r\}$, define the two potentially
excluded canonical residues $x_i(n)$ and $y_i(n)$ by
\begin{equation}
\label{eq:xy-definition}
\begin{array}{c|c|c}
e(n) & x_i(n) & y_i(n)\\
\hline
0
&
a_i
&
\delta_{p_i}\bigl(m(n)-b_i\bigr)
\\[4pt]
1
&
a_i
&
\delta_{p_i}\bigl(m(n)-a_i\bigr)
\\[4pt]
2
&
b_i
&
\delta_{p_i}\bigl(m(n)-b_i\bigr)
\end{array}
\end{equation}

Indeed, Lemma~\ref{lem:local-zero} and the parametrization
\eqref{eq:parametrization-table} give, for every
$z\in\Zcal(n)$,
\begin{equation}
\label{eq:local-exclusion}
p_i\mid H_z(n)K_z(n)
\quad\Longleftrightarrow\quad
\delta_{p_i}(z)\in\{x_i(n),y_i(n)\}.
\end{equation}

Thus each prime $p_i$ excludes at most two canonical residue
classes of the parameter $z$. These two classes need not always
be distinct.

\begin{lemma}[Collision lemma]
\label{lem:collision}
For every $i\in\{1,\dots,r\}$,
\begin{equation}
\label{eq:collision}
x_i(n)=y_i(n)
\quad\Longleftrightarrow\quad
\delta_{p_i}(n)=0.
\end{equation}
\end{lemma}

\begin{proof}
Suppose first that $e(n)=0$. By
\eqref{eq:xy-definition}, equality of the two canonical residues
is equivalent to
\[
\delta_{p_i}(a_i)
=
\delta_{p_i}\bigl(m(n)-b_i\bigr).
\]
Using \eqref{eq:delta-divisibility} and
$a_i+b_i=p_i-1$, this is equivalent to
\[
p_i\mid m(n)+1.
\]
Since
\[
n=3\bigl(m(n)+1\bigr)
\]
and $p_i\neq3$, the last condition is equivalent to
$\delta_{p_i}(n)=0$.

Suppose now that $e(n)=1$. Equality of the excluded residues is
equivalent to
\[
p_i\mid m(n)-2a_i.
\]
Since $n=3m(n)+1$,
\[
3\bigl(m(n)-2a_i\bigr)
=
n-\bigl(6a_i+1\bigr).
\]
By \eqref{eq:a-identity}, the quantity $6a_i+1$ is divisible by
$p_i$. Hence
\[
p_i\mid m(n)-2a_i
\quad\Longleftrightarrow\quad
p_i\mid n.
\]

Finally, suppose that $e(n)=2$. Equality is now equivalent to
\[
p_i\mid m(n)-2b_i.
\]
Because $n=3m(n)+5$,
\[
3\bigl(m(n)-2b_i\bigr)
=
n-\bigl(6b_i+5\bigr).
\]
The quantity $6b_i+5$ is divisible by $p_i$ by
\eqref{eq:b-identity}. Therefore,
\[
p_i\mid m(n)-2b_i
\quad\Longleftrightarrow\quad
p_i\mid n.
\]
This completes the proof.
\end{proof}

Define the collision indicator
\begin{equation}
\label{eq:epsilon-definition}
\varepsilon_i(n)
=
\one{\delta_{p_i}(n)=0}.
\end{equation}

The local admissibility indicator can then be written uniformly
as
\begin{equation}
\label{eq:local-indicator}
\one{
\delta_{p_i}(z)\notin\{x_i(n),y_i(n)\}
}
=
1
-
\bigl(1-\varepsilon_i(n)\bigr)
\one{\delta_{p_i}(z)=x_i(n)}
-
\one{\delta_{p_i}(z)=y_i(n)}.
\end{equation}

If $x_i(n)\neq y_i(n)$, then $\varepsilon_i(n)=0$ and the
right-hand side removes both excluded classes. If
$x_i(n)=y_i(n)$, then $\varepsilon_i(n)=1$ and the common class
is removed exactly once.

\subsection{CRT expansion and the exact formula}

We first record the elementary interval count used in the
expansion.

\begin{lemma}[Canonical interval count]
\label{lem:L-count}
Let $m\geq-1$, let $q\geq1$, and let
$0\leq a\leq q-1$. Then
\begin{equation}
\label{eq:L-definition}
L(m,q,a)
:=
\#\left\{
z\in\Z
\mathrel{}\middle|\mathrel{}
0\leq z\leq m,\ 
\delta_q(z)=a
\right\}
=
1+\left\lfloor\frac{m-a}{q}\right\rfloor.
\end{equation}
\end{lemma}

\begin{proof}
If $a\leq m$, the integers being counted are
\[
a,\ a+q,\ a+2q,\dots
\]
up to $m$, which gives the stated formula.

If $a>m$, then
\[
-1\leq\frac{m-a}{q}<0.
\]
Consequently,
\[
1+\left\lfloor\frac{m-a}{q}\right\rfloor=0.
\]
This includes $m=-1$, for which the interval
$\{z\in\Z:0\leq z\leq m\}$ is empty.
\end{proof}

Let
\[
\boldsymbol\alpha
=
(\alpha_1,\dots,\alpha_r)
\in\{0,1,2\}^{r}
\]
and define its active index set by
\begin{equation}
\label{eq:active-indices}
I_{\boldsymbol\alpha}
=
\left\{
i\in\{1,\dots,r\}
\mathrel{}\middle|\mathrel{}
\alpha_i\neq0
\right\}.
\end{equation}

For every $i$, define the local weights
\begin{equation}
\label{eq:local-weights}
w_{i,0}(n)=1,
\qquad
w_{i,1}(n)=-(1-\varepsilon_i(n)),
\qquad
w_{i,2}(n)=-1,
\end{equation}
and put
\begin{equation}
\label{eq:global-weight}
W_{\boldsymbol\alpha}(n)
=
\prod_{i=1}^{r}w_{i,\alpha_i}(n).
\end{equation}

The modulus associated with $\boldsymbol\alpha$ is
\begin{equation}
\label{eq:Q-alpha}
Q_{\boldsymbol\alpha}
=
\prod_{i\in I_{\boldsymbol\alpha}}p_i.
\end{equation}
When $I_{\boldsymbol\alpha}=\varnothing$, the product is
understood to be $1$.

For $\ell\in\{1,2\}$, set
\begin{equation}
\label{eq:s-definition}
s_{i,1}(n)=x_i(n),
\qquad
s_{i,2}(n)=y_i(n).
\end{equation}

If $I_{\boldsymbol\alpha}=\varnothing$, define
\[
\eta_{\boldsymbol\alpha}(n)=0.
\]
If $I_{\boldsymbol\alpha}\neq\varnothing$, the Chinese remainder
theorem gives a unique
\[
\eta_{\boldsymbol\alpha}(n)
\in
\{0,\dots,Q_{\boldsymbol\alpha}-1\}
\]
satisfying
\begin{equation}
\label{eq:eta-crt}
\delta_{p_i}\bigl(\eta_{\boldsymbol\alpha}(n)\bigr)
=
s_{i,\alpha_i}(n),
\qquad
i\in I_{\boldsymbol\alpha}.
\end{equation}

For completeness, this canonical representative can be written
explicitly. For every $i\in I_{\boldsymbol\alpha}$, let
\[
Q_{\boldsymbol\alpha,i}
=
\frac{Q_{\boldsymbol\alpha}}{p_i}
\]
and let
\[
u_{\boldsymbol\alpha,i}
\in\{0,\dots,p_i-1\}
\]
be the unique canonical residue satisfying
\[
\delta_{p_i}
\left(
Q_{\boldsymbol\alpha,i}
u_{\boldsymbol\alpha,i}
\right)
=
1.
\]
Then
\begin{equation}
\label{eq:eta-explicit}
\eta_{\boldsymbol\alpha}(n)
=
\delta_{Q_{\boldsymbol\alpha}}
\left(
\sum_{i\in I_{\boldsymbol\alpha}}
s_{i,\alpha_i}(n)
Q_{\boldsymbol\alpha,i}
u_{\boldsymbol\alpha,i}
\right).
\end{equation}

Define
\begin{equation}
\label{eq:C-definition}
\Ccal_P(n)
=
\#\left\{
z\in\Zcal(n)
\mathrel{}\middle|\mathrel{}
\gcd\bigl(H_z(n)K_z(n),P\bigr)=1
\right\}
\end{equation}
and
\begin{equation}
\label{eq:F-definition}
F_P(n)
=
\one{\gcd(n,6P)=1}.
\end{equation}
Thus $F_P(n)$ is the indicator that the central decomposition
$2n=n+n$ is admissible.

\begin{theorem}[Exact CRT formula]
\label{thm:exact-formula}
For every positive integer $n$,
\begin{equation}
\label{eq:C-exact}
\Ccal_P(n)
=
\sum_{\boldsymbol\alpha\in\{0,1,2\}^{r}}
W_{\boldsymbol\alpha}(n)
L\left(
m(n),
Q_{\boldsymbol\alpha},
\eta_{\boldsymbol\alpha}(n)
\right).
\end{equation}
Moreover,
\begin{equation}
\label{eq:g-exact}
g_P(2n)
=
\begin{cases}
\Ccal_P(n),
&
\delta_3(n)=0,
\\[5pt]
\dfrac{\Ccal_P(n)+F_P(n)}{2},
&
\delta_3(n)\in\{1,2\}.
\end{cases}
\end{equation}
\end{theorem}

\begin{proof}
For $z\in\Zcal(n)$, define
\[
I_i(z)
=
\one{
\delta_{p_i}(z)\notin\{x_i(n),y_i(n)\}
}.
\]
By \eqref{eq:local-exclusion},
\[
\prod_{i=1}^{r}I_i(z)=1
\]
exactly when
\[
\gcd\bigl(H_z(n)K_z(n),P\bigr)=1.
\]
Therefore,
\begin{equation}
\label{eq:C-indicator-sum}
\Ccal_P(n)
=
\sum_{z\in\Zcal(n)}
\prod_{i=1}^{r}I_i(z).
\end{equation}

Expanding each factor by means of
\eqref{eq:local-indicator} gives
\[
\prod_{i=1}^{r}I_i(z)
=
\sum_{\boldsymbol\alpha\in\{0,1,2\}^{r}}
W_{\boldsymbol\alpha}(n)
\one{
\delta_{Q_{\boldsymbol\alpha}}(z)
=
\eta_{\boldsymbol\alpha}(n)
}.
\]
For $\boldsymbol\alpha=\boldsymbol0$, this indicator is
identically $1$ because
\[
Q_{\boldsymbol0}=1,
\qquad
\eta_{\boldsymbol0}(n)=0,
\qquad
\delta_1(z)=0.
\]

Summing the expansion over $z\in\Zcal(n)$ and applying
Lemma~\ref{lem:L-count} gives \eqref{eq:C-exact}.

Suppose that $\delta_3(n)=0$. By
Proposition~\ref{prop:six-wheel}, every admissible parameter
corresponds to exactly one unordered representation. Hence
\[
g_P(2n)=\Ccal_P(n).
\]

Suppose now that $\delta_3(n)\in\{1,2\}$. The involution
\[
z\longmapsto m(n)-z
\]
preserves admissibility because it interchanges $H_z(n)$ and
$K_z(n)$. Burnside's lemma gives
\[
g_P(2n)
=
\frac{\Ccal_P(n)+F_P(n)}{2},
\]
provided that $F_P(n)$ is exactly the number of admissible fixed
parameters. We now verify this identification.

A fixed parameter exists exactly when $m(n)$ is even. In the two
cases under consideration,
\[
n=3m(n)+1
\qquad\text{or}\qquad
n=3m(n)+5.
\]
Since both constant terms are odd, $m(n)$ is even exactly when
$n$ is odd. If $\gcd(n,6P)=1$, then $n$ is odd, so the fixed
parameter
\[
z=\frac{m(n)}2
\]
exists and gives $H_z(n)=K_z(n)=n$. It is admissible. Conversely,
every admissible fixed parameter gives the decomposition
$2n=n+n$ and therefore requires $\gcd(n,6P)=1$. Hence the number of admissible fixed parameters is exactly
$F_P(n)$, which proves \eqref{eq:g-exact}.
\end{proof}

\begin{corollary}[One-prime specialization]
\label{cor:r-one}
If $r=1$ and $P=p$, then
\begin{equation}
\label{eq:r-one-C}
\Ccal_p(n)
=
m(n)+1
-
\bigl(1-\varepsilon_1(n)\bigr)
L\bigl(m(n),p,x_1(n)\bigr)
-
L\bigl(m(n),p,y_1(n)\bigr).
\end{equation}
Together with \eqref{eq:g-exact}, this specializes the general
formula to an equivalent form of the one-prime result in
\cite{Salazar2026}.
\end{corollary}

\begin{proof}
When $r=1$, the three possible values of $\alpha_1$ are
$0$, $1$, and $2$, with respective weights
\[
1,
\qquad
-(1-\varepsilon_1(n)),
\qquad
-1.
\]
The term corresponding to $\alpha_1=0$ is
\[
L(m(n),1,0)=m(n)+1.
\]
\end{proof}

\begin{remark}
For fixed $P$, the inverses in \eqref{eq:eta-explicit} may be
precomputed.  Since $w_{i,1}(n)=0$ exactly when
$\varepsilon_i(n)=1$, the number of terms with nonzero weight in
\eqref{eq:C-exact} is
\[
\prod_{i=1}^{r}\bigl(3-\varepsilon_i(n)\bigr),
\]
and is therefore at most $3^r$. Some of these summands can still
vanish because their interval counts are zero. The expansion is a finite
Legendre-type inclusion--exclusion formula in which the surviving
residue classes are represented explicitly by the Chinese
remainder theorem; compare the classical Legendre sieve
\cite{FriedlanderIwaniec2010}.
\end{remark}

\section{Finite local correlations and affine periodicity}

The exact CRT formula has a complementary complete-modulus
interpretation.  The classical Nagell totient, as studied by Cohen
\cite{Cohen1960}, is
\[
\theta(a,m)
=
\#\left\{
1\leq c\leq m:
\gcd(c,m)=\gcd(a-c,m)=1
\right\},
\]
for positive integers $a$ and $m$.  In the present setting it is
natural to use canonical residues and define
\begin{equation}
\label{eq:R-definition}
\Rcal_P(n)
=
\left\{
c\in\{0,\dots,M-1\}:
\gcd(c,M)=
\gcd\bigl(\delta_M(2n-c),M\bigr)=1
\right\}.
\end{equation}
Thus $\#\Rcal_P(n)=\theta(2n,M)$.

The reflection
\[
\rho_n(c)=\delta_M(2n-c)
\]
is an involution of $\Rcal_P(n)$.  Define
\begin{equation}
\label{eq:A-definition}
A_P(n)
=
\prod_{i=1}^{r}\left(p_i-2+\varepsilon_i(n)\right)
\end{equation}
and
\begin{equation}
\label{eq:kappa-definition}
\kappa_P(n)
=
\left(1+\one{\delta_3(n)=0}\right)A_P(n).
\end{equation}

\begin{proposition}[Nagell correlation for the squarefree modulus]
\label{prop:R-cardinality}
For every positive integer $n$,
\begin{equation}
\label{eq:R-cardinality}
\theta(2n,M)
=
\#\Rcal_P(n)
=
\kappa_P(n).
\end{equation}
\end{proposition}

\begin{proof}
Let $q$ be a prime divisor of $M$.  The two local coprimality
conditions exclude the residues
\[
0
\qquad\text{and}\qquad
\delta_q(2n)
\]
under $\delta_q$. They coincide exactly when $q\mid2n$.  Hence the
number of allowed residues is $q-1$ in the collision case and
$q-2$ otherwise.  The factor at $q=2$ is $1$, the factor at
$q=3$ is $1+\one{\delta_3(n)=0}$, and the factor at
$q=p_i$ is $p_i-2+\varepsilon_i(n)$.  Since $M$ is squarefree,
the Chinese remainder theorem gives the product
\eqref{eq:kappa-definition}.
\end{proof}

For every prime $q\mid M$, put
\[
\nu_q(2n)=2-\one{\delta_q(2n)=0}
\]
and define the finite singular factor
\begin{equation}
\label{eq:finite-singular-factor}
\mathfrak S_M^{\mathrm{fin}}(2n)
=
\prod_{q\mid M}
\frac{1-\nu_q(2n)/q}{(1-1/q)^2}.
\end{equation}

The product \eqref{eq:finite-singular-factor} consists of the local
factors, at the primes dividing $M$, of the Goldbach singular series
of Hardy and Littlewood \cite{HardyLittlewood1923}. When these are
all primes up to a cutoff it is an initial finite Euler product;
for a general $P$ it is a product over the specified prime support.

\begin{corollary}[Finite singular-factor identity]
\label{cor:singular-factor}
One has
\begin{equation}
\label{eq:kappa-singular}
\frac{\kappa_P(n)}{M}
=
\left(\frac{\varphi(M)}{M}\right)^2
\mathfrak S_M^{\mathrm{fin}}(2n).
\end{equation}
\end{corollary}

\begin{proof}
By Proposition~\ref{prop:R-cardinality},
\[
\frac{\kappa_P(n)}{M}
=
\prod_{q\mid M}
\left(1-\frac{\nu_q(2n)}q\right),
\]
whereas
\[
\frac{\varphi(M)}M
=
\prod_{q\mid M}\left(1-\frac1q\right).
\]
Factoring the square of the latter local term proves the identity.
\end{proof}

The periodic lift of the finite correlation is
\begin{equation}
\label{eq:R-lift}
\widetilde{\Rcal}_P(n)
=
\left\{
a\in\Z:
\gcd(a,M)=\gcd(2n-a,M)=1
\right\}.
\end{equation}
It contains exactly $\kappa_P(n)$ integers in every interval of
$M$ consecutive integers.

\begin{proposition}[Reflection identity]
\label{prop:reflection-identity}
For every positive integer $n$,
\begin{equation}
\label{eq:reflection-identity}
2g_P(2n)-F_P(n)
=
\#
\left(
\widetilde{\Rcal}_P(n)\cap[1,2n]
\right).
\end{equation}
Consequently,
\begin{equation}
\label{eq:reflection-block-bound}
g_P(2n)
\geq
\frac12
\left\lfloor\frac{2n}{M}\right\rfloor
\kappa_P(n).
\end{equation}
\end{proposition}

\begin{proof}
Neither $0$ nor $2n$ belongs to
$\widetilde{\Rcal}_P(n)$.  The involution
$a\mapsto2n-a$ pairs the remaining elements in
$[1,2n]$, except for the fixed point $a=n$, which is present
exactly when $F_P(n)=1$.  The resulting orbits are precisely the
unordered admissible decompositions counted by $g_P(2n)$, proving
\eqref{eq:reflection-identity}.

The interval $[1,2n]$ contains
$\lfloor2n/M\rfloor$ disjoint complete blocks of length $M$.
Each contains $\kappa_P(n)$ elements of the lift.  Applying
\eqref{eq:reflection-identity} and discarding the nonnegative
term $F_P(n)$ gives \eqref{eq:reflection-block-bound}.
\end{proof}

\subsection{Affine periodicity}

\begin{proposition}[Half-period identity]
\label{prop:half-period-identity}
For every positive integer $n$,
\begin{equation}
\label{eq:half-period-identity}
\begin{aligned}
2g_P\bigl(2(n+3P)\bigr)-F_P(n+3P)
&=2g_P(2n)-F_P(n)+\kappa_P(n).
\end{aligned}
\end{equation}
\end{proposition}
\begin{proof}
The shift $n\mapsto n+3P$ changes $2n$ by $M=6P$.
Thus the periodic lifts are identical, while the interval in
\eqref{eq:reflection-identity} gains exactly one complete block of
$M$ integers. This block contains $\kappa_P(n)$ admissible integers.
Subtracting the two reflection identities proves the formula.
\end{proof}

\begin{theorem}[Affine periodicity]
\label{thm:affine-periodicity}
For every positive integer $n$,
\begin{equation}
\label{eq:affine-periodicity}
g_P\bigl(2(n+M)\bigr)-g_P(2n)
=
\kappa_P(n).
\end{equation}
\end{theorem}

\begin{proof}
The sets
$\widetilde{\Rcal}_P(n+M)$ and $\widetilde{\Rcal}_P(n)$ are
identical, because $\delta_M(2M)=0$.  Enlarging
$[1,2n]$ to $[1,2n+2M]$ adds two complete periods and therefore
$2\kappa_P(n)$ elements.  Moreover
$F_P(n+M)=F_P(n)$.  Subtracting the two instances of
\eqref{eq:reflection-identity} gives the result.
\end{proof}

\begin{corollary}[Quasipolynomial form]
\label{cor:quasipolynomial}
Write uniquely
\[
n=a+tM,
\qquad
1\leq a\leq M,
\qquad
t\in\Z_{\geq0}.
\]
Then
\begin{equation}
\label{eq:quasipolynomial}
g_P(2n)
=
g_P(2a)+t\kappa_P(a).
\end{equation}
Hence $M=6P$ is a quasiperiod of the degree-one
quasipolynomial $g_P(2n)$.
\end{corollary}

\begin{proof}
Iterate Theorem~\ref{thm:affine-periodicity} and use
$\kappa_P(a+jM)=\kappa_P(a)$.
\end{proof}

\subsection{The subsequence \texorpdfstring{$n=tP$}{n=tP}}

Define the nonprincipal real Dirichlet character of period $3$ by
\begin{equation}
\label{eq:chi3-definition}
\chi_3(a)
=
\begin{cases}
0, & \delta_3(a)=0,\\
1, & \delta_3(a)=1,\\
-1, & \delta_3(a)=2.
\end{cases}
\end{equation}

\begin{theorem}[Exact formula on the ray $n=tP$]
\label{thm:ray-formula}
For every positive integer $t$,
\begin{equation}
\label{eq:ray-formula}
g_P(2tP)
=
\frac{
\left(1+\one{\delta_3(t)=0}\right)t\varphi(P)
}{6}
+
\frac{\chi_3(t)}{3}
\prod_{i=1}^{r}
\bigl(\chi_3(p_i)-1\bigr).
\end{equation}
\end{theorem}

\begin{proof}
Put $n=tP$. Formula \eqref{eq:g-short-interval} gives
\[
g_P(2tP)
=
\#\left\{
1\leq h\leq tP
\mathrel{}\middle|\mathrel{}
\gcd(h,M)=\gcd(2tP-h,M)=1
\right\}.
\]

For every prime $p_i\mid P$,
\[
p_i\mid2tP.
\]
Therefore,
\[
p_i\mid2tP-h
\quad\Longleftrightarrow\quad
p_i\mid h.
\]
The same equivalence holds for the prime $2$. Only the
restriction contributed by the prime $3$ depends on
$\delta_3(t)$.

Suppose first that $\delta_3(t)=0$. Write $t=3u$. Then
\[
2tP=uM,
\]
so
\[
\gcd(2tP-h,M)=\gcd(h,M).
\]
Consequently, $g_P(2tP)$ is the number of integers coprime to
$M$ in the interval
\[
1\leq h\leq tP=\frac{uM}{2}.
\]

The reduced residues under $\delta_M$ occur in pairs
\[
c
\qquad\text{and}\qquad
M-c.
\]
There is no fixed point: $M/2=3P$ is not coprime to $M$.
Exactly one member of each pair lies in
$[1,M/2]$. Hence
\[
\#\left\{
1\leq h\leq \frac M2
\mathrel{}\middle|\mathrel{}
\gcd(h,M)=1
\right\}
=
\frac{\varphi(M)}2
=
\varphi(P).
\]
Here $\varphi(M)=2\varphi(P)$ because $M=6P$ and
$\gcd(P,6)=1$.
Dividing the interval $[1,uM/2]$ into complete periods and,
when necessary, one half-period gives
\[
g_P(2tP)
=
\frac{u\varphi(M)}2
=
u\varphi(P)
=
\frac{t\varphi(P)}3.
\]
This is the first term of \eqref{eq:ray-formula}, while the
character term vanishes because $\chi_3(t)=0$.

Suppose now that $\delta_3(t)\neq0$. Since $P$ is coprime to
$3$, one also has $\delta_3(tP)\neq0$. The restrictions at the
primes $2$ and $3$ are equivalent to
\begin{equation}
\label{eq:ray-local-conditions}
\delta_2(h)=1,
\qquad
\delta_3(h)=\delta_3(tP).
\end{equation}
Indeed, the first condition says that $h$ is odd. Under
$\delta_3$, the two excluded residues are $0$ and
$\delta_3(2tP)$. If $\delta_3(tP)=1$, the remaining residue is
$1$, and if $\delta_3(tP)=2$, the remaining residue is $2$.

Let $\mu$ denote the M\"obius function. Applying M\"obius inversion
to the condition $\gcd(h,P)=1$ gives
\[
g_P(2tP)
=
\sum_{d\mid P}\mu(d)
\#
\left\{
1\leq h\leq tP
\mathrel{}\middle|\mathrel{}
d\mid h,\
\delta_2(h)=1,\
\delta_3(h)=\delta_3(tP)
\right\}.
\]

Write
\[
h=d\ell,
\qquad
X_d=\frac{tP}{d}.
\]
Every divisor $d$ of $P$ is coprime to $6$. In particular, $d$
is odd, so the condition under $\delta_2$ becomes
$\delta_2(\ell)=1$. Moreover, $tP=dX_d$, and multiplication by
$d$ is invertible under $\delta_3$. The second condition in
\eqref{eq:ray-local-conditions} therefore becomes
\[
\delta_3(\ell)=\delta_3(X_d).
\]

For every positive integer $X$ satisfying $\delta_3(X)\neq0$,
define
\[
R_6(X)
=
\#
\left\{
1\leq\ell\leq X
\mathrel{}\middle|\mathrel{}
\delta_2(\ell)=1,\
\delta_3(\ell)=\delta_3(X)
\right\}.
\]
The four possible canonical residues of $X$ under $\delta_6$
give
\[
\begin{array}{c|c|c}
\delta_6(X)
& \text{admissible residue of $\ell$ under $\delta_6$}
& R_6(X)\\
\hline
1 & 1 & \dfrac{X+5}{6}\\[4pt]
2 & 5 & \dfrac{X-2}{6}\\[4pt]
4 & 1 & \dfrac{X+2}{6}\\[4pt]
5 & 5 & \dfrac{X+1}{6}
\end{array}
\]
and hence
\begin{equation}
\label{eq:ray-class-count}
R_6(X)
=
\frac X6
+\frac12\one{\delta_2(X)=1}
+\frac{\chi_3(X)}3.
\end{equation}

Since $P/d$ is odd for every $d\mid P$, one has
\[
\delta_2(X_d)=\delta_2(t).
\]
Substitution of \eqref{eq:ray-class-count} therefore gives
\[
g_P(2tP)
=
\sum_{d\mid P}\mu(d)
\left(
\frac{tP}{6d}
+\frac12\one{\delta_2(t)=1}
+\frac{\chi_3(tP/d)}3
\right).
\]

Since $P>1$,
\[
\sum_{d\mid P}\mu(d)=0.
\]
Also,
\[
\sum_{d\mid P}\mu(d)\frac{P}{d}
=
\varphi(P),
\]
and the multiplicativity of $\chi_3$ gives
\[
\sum_{d\mid P}
\mu(d)\chi_3(tP/d)
=
\chi_3(t)
\prod_{i=1}^{r}
\bigl(\chi_3(p_i)-1\bigr).
\]
Combining these three identities proves
\eqref{eq:ray-formula}.
\end{proof}

\begin{remark}
For each $i$, the factor $\chi_3(p_i)-1$ is $-2$ when
$\delta_3(p_i)=2$ and $0$ when $\delta_3(p_i)=1$.  Hence the
second term in \eqref{eq:ray-formula} either vanishes or has
absolute value $2^r/3$. If some $p_i$ satisfies $\delta_3(p_i)=1$,
the character correction vanishes for every $t$. If all satisfy
$\delta_3(p_i)=2$, the correction is $(-2)^r\chi_3(t)/3$.
\end{remark}

\section{Finite density and positivity thresholds}

\subsection{Parameter-density decomposition}

Recall
\[
A_P(n)
=
\prod_{i=1}^{r}\left(p_i-2+\varepsilon_i(n)\right)
\]
and
\[
\kappa_P(n)
=
\left(1+\one{\delta_3(n)=0}\right)A_P(n).
\]
Define
\begin{equation}
\label{eq:B-discrepancy}
\Bcal_P(n)
=
\sum_{\substack{\boldsymbol\alpha\in\{0,1,2\}^r\\
\boldsymbol\alpha\neq\boldsymbol0}}
W_{\boldsymbol\alpha}(n)
\left[
L\left(
m(n),Q_{\boldsymbol\alpha},
\eta_{\boldsymbol\alpha}(n)
\right)
-
\frac{m(n)+1}{Q_{\boldsymbol\alpha}}
\right].
\end{equation}

\begin{proposition}[Parameter-density decomposition]
\label{prop:C-density}
For every positive integer $n$,
\begin{equation}
\label{eq:C-density}
\Ccal_P(n)
=
\frac{m(n)+1}{P}A_P(n)+\Bcal_P(n),
\end{equation}
and
\begin{equation}
\label{eq:B-bound}
\left|\Bcal_P(n)\right|
<
\prod_{i=1}^{r}
\bigl(3-\varepsilon_i(n)\bigr)-1
\leq3^r-1.
\end{equation}
\end{proposition}

\begin{proof}
From the local weights,
\[
\sum_{\boldsymbol\alpha\in\{0,1,2\}^r}
\frac{W_{\boldsymbol\alpha}(n)}{Q_{\boldsymbol\alpha}}
=
\prod_{i=1}^{r}
\left(
1-\frac{1-\varepsilon_i(n)}{p_i}-\frac1{p_i}
\right)
=
\frac{A_P(n)}P.
\]
Separating the mean term from each summand of
\eqref{eq:C-exact} gives \eqref{eq:C-density}.

For a nonzero $\boldsymbol\alpha$, write
\[
m(n)+1=vQ_{\boldsymbol\alpha}+s,
\qquad
0\leq s<Q_{\boldsymbol\alpha}.
\]
If $s=0$, every residue class occurs exactly $v$ times in the
parameter interval, so its discrepancy from the mean is zero.
If $s>0$, each residue class occurs either $v$ or $v+1$ times,
and its discrepancy from the mean has absolute value smaller than
$1$.  Therefore
\[
\left|
L\left(m(n),Q_{\boldsymbol\alpha},
\eta_{\boldsymbol\alpha}(n)\right)
-
\frac{m(n)+1}{Q_{\boldsymbol\alpha}}
\right|
<1.
\]
Finally,
\[
\sum_{\boldsymbol\alpha\neq\boldsymbol0}
|W_{\boldsymbol\alpha}(n)|
=
\prod_{i=1}^{r}\bigl(3-\varepsilon_i(n)\bigr)-1,
\]
and the triangle inequality proves \eqref{eq:B-bound}.
\end{proof}

\begin{proposition}[Refined discrepancy bound]
\label{prop:refined-discrepancy}
For every positive integer $n$, put
\[
D_P^{\mathrm{ref}}(n)=
\prod_{i=1}^{r}(3-\varepsilon_i(n))
-\prod_{i=1}^{r}\left(1+\frac{2-\varepsilon_i(n)}{p_i}\right).
\]
Then
\begin{equation}
\label{eq:refined-B-bound}
|\Bcal_P(n)|\leq D_P^{\mathrm{ref}}(n)
<\prod_{i=1}^{r}(3-\varepsilon_i(n))-1.
\end{equation}
\end{proposition}
\begin{proof}
Write $m+1=vq+s$, where $0\leq s<q$. If $s=0$, the discrepancy
of each residue count is zero. Otherwise it is either $-s/q$ or
$1-s/q$. Hence, for all integers $m\geq-1$, $q\geq1$, and
$0\leq a<q$,
\[
\left|L(m,q,a)-\frac{m+1}{q}\right|\leq1-\frac1q.
\]
Consequently, the triangle inequality gives
\[
|\Bcal_P(n)|\leq
\sum_{\boldsymbol\alpha}|W_{\boldsymbol\alpha}(n)|
\left(1-\frac1{Q_{\boldsymbol\alpha}}\right)
=D_P^{\mathrm{ref}}(n).
\]
The equality follows by factoring the two sums of absolute weights.
The term with $\boldsymbol\alpha=\boldsymbol0$ contributes zero.
Since $r\geq1$, the second product defining $D_P^{\mathrm{ref}}(n)$
is strictly greater than $1$, proving the strict inequality.
\end{proof}

The corresponding decomposition of the representation count is
still useful for comparison with the finite local density.

\begin{theorem}[Finite-density decomposition]
\label{thm:density-formula}
For every positive integer $n$,
\begin{equation}
\label{eq:density-formula}
g_P(2n)
=
\frac{n}{M}\kappa_P(n)+\Ecal_P(n),
\qquad M=6P,
\end{equation}
where
\begin{equation}
\label{eq:E-branches}
\Ecal_P(n)
=
\begin{cases}
\Bcal_P(n),
&
\delta_3(n)=0,
\\[5pt]
\dfrac{\Bcal_P(n)+F_P(n)}{2}
+\dfrac{A_P(n)}{3P},
&
\delta_3(n)=1,
\\[9pt]
\dfrac{\Bcal_P(n)+F_P(n)}{2}
-\dfrac{A_P(n)}{3P},
&
\delta_3(n)=2.
\end{cases}
\end{equation}
In particular,
\begin{equation}
\label{eq:E-bound}
|\Ecal_P(n)|<3^r-1.
\end{equation}
\end{theorem}

\begin{proof}
If $\delta_3(n)=0$, then
$m(n)+1=n/3$, $\kappa_P(n)=2A_P(n)$, and
$g_P(2n)=\Ccal_P(n)$, which gives the first branch.

If $\delta_3(n)=1$, then
$m(n)+1=(n+2)/3$ and $\kappa_P(n)=A_P(n)$.  Combining
\eqref{eq:g-exact} and \eqref{eq:C-density} gives the second
branch.  The case $\delta_3(n)=2$ is identical with
$m(n)+1=(n-2)/3$ and gives the third branch.

The bound follows from \eqref{eq:B-bound}.  In the two latter
cases,
\[
|\Ecal_P(n)|
<
\frac{3^r-1}{2}
+\frac12
+\frac13
<
3^r-1,
\]
because $0<A_P(n)/P<1$ and $r\geq1$.
\end{proof}

\begin{example}
For $P=35$ and $n=20$ one has
\[
g_P(40)=2,
\qquad
\frac{n\kappa_P(n)}{6P}=\frac{40}{21},
\qquad
\Ecal_P(n)=\frac{2}{21}.
\]
The two admissible representations are
$40=11+29$ and $40=17+23$. Explicitly, $m(20)=5$ and the excluded
residues under $\delta_5$ are $x_1=y_1=0$, while under $\delta_7$
they are $x_2=5$, $y_2=0$. The CRT expansion reduces to
\[
\begin{aligned}
\Ccal_{35}(20)
&=6-L(5,5,0)-L(5,7,5)-L(5,7,0)\\
&\quad+L(5,35,5)+L(5,35,0)\\
&=6-2-1-1+1+1=4.
\end{aligned}
\]
Since $F_{35}(20)=0$, the unordered count is $4/2=2$.
For $P=385$ and $n=14$,
\[
g_P(28)=0,
\qquad
\frac{n\kappa_P(n)}{6P}=\frac{54}{55},
\qquad
\Ecal_P(n)=-\frac{54}{55}.
\]
Thus a positive local density can coexist with an empty truncated
interval.
\end{example}

\subsection{Coverings and a unified positivity criterion}

For $1\leq i\leq r$, let
\[
X_i(n)
=
\{z\in\Zcal(n):\delta_{p_i}(z)=x_i(n)\},
\qquad
Y_i(n)
=
\{z\in\Zcal(n):\delta_{p_i}(z)=y_i(n)\}.
\]

\begin{proposition}[Covering equivalence]
\label{prop:covering}
For every positive integer $n$,
\begin{equation}
\label{eq:covering-equivalence}
g_P(2n)=0
\quad\Longleftrightarrow\quad
\Ccal_P(n)=0
\quad\Longleftrightarrow\quad
\Zcal(n)
\subseteq
\bigcup_{i=1}^{r}\bigl(X_i(n)\cup Y_i(n)\bigr).
\end{equation}
\end{proposition}

\begin{proof}
The last equivalence follows directly from
\eqref{eq:local-exclusion}.  If $\delta_3(n)=0$, the first
equivalence follows from $g_P(2n)=\Ccal_P(n)$.  If
$\delta_3(n)\in\{1,2\}$, then
$g_P(2n)=(\Ccal_P(n)+F_P(n))/2$.  Whenever $F_P(n)=1$, the
central parameter $z=m(n)/2$ is admissible, so
$\Ccal_P(n)\geq1$.  Hence $\Ccal_P(n)=0$ forces $F_P(n)=0$,
which proves the first equivalence.
\end{proof}

\begin{theorem}[Parameter-density positivity criterion]
\label{thm:parameter-density-positive}
If
\begin{equation}
\label{eq:parameter-density-criterion}
\frac{m(n)+1}{P}A_P(n)
\geq
\prod_{i=1}^{r}\bigl(3-\varepsilon_i(n)\bigr)-1,
\end{equation}
then
\[
g_P(2n)>0.
\]
\end{theorem}

\begin{proof}
By Proposition~\ref{prop:C-density}, the strict bound
\eqref{eq:B-bound} and hypothesis
\eqref{eq:parameter-density-criterion} give
$\Ccal_P(n)>0$.  Proposition~\ref{prop:covering} then gives the
result.
\end{proof}

\begin{remark}[Collisions and the refined criterion]
\label{rem:refined-positive}
If $c_P(n)=\#\{i:\delta_{p_i}(n)=0\}$, then
\[
\prod_i(3-\varepsilon_i(n))-1=2^{c_P(n)}3^{r-c_P(n)}-1.
\]
Proposition~\ref{prop:refined-discrepancy} also gives the sufficient condition
\[
\frac{m(n)+1}{P}A_P(n)>D_P^{\mathrm{ref}}(n)
\quad\Longrightarrow\quad g_P(2n)>0.
\]
The strict inequality is required for this inference, because
\eqref{eq:refined-B-bound} can be an equality. For example, when
$P=5$ and $n=10$, one has $\Bcal_5(10)=D_5^{\mathrm{ref}}(10)=4/5$.
\end{remark}

\begin{proposition}[Half-period positivity]
\label{cor:half-period}
For every $n\geq3P$,
\begin{equation}
\label{eq:half-period-quantitative}
g_P(2n)\geq\frac12\kappa_P(n)
\geq\frac12\prod_{i=1}^{r}(p_i-2)>0.
\end{equation}
\end{proposition}
\begin{proof}
Since $2n\geq M$, one has $\lfloor2n/M\rfloor\geq1$.
Apply \eqref{eq:reflection-block-bound} and the definition of
$\kappa_P(n)$.
\end{proof}

\begin{corollary}[Residue-free positivity threshold]
\label{cor:residue-free-threshold}
Put
\begin{equation}
\label{eq:T-star-definition}
T^*(P)
=
3(3^r-1)
\prod_{i=1}^{r}\frac{p_i}{p_i-2}
+2.
\end{equation}
If $n\geq T^*(P)$, then $g_P(2n)>0$.
Moreover,
\begin{equation}
\label{eq:T-star-half-period}
T^*(P)\leq3P.
\end{equation}
\end{corollary}

\begin{proof}
For every $n$,
\[
m(n)+1\geq\frac{n-2}{3},
\qquad
A_P(n)\geq\prod_{i=1}^{r}(p_i-2),
\]
while
$\prod_i(3-\varepsilon_i(n))-1\leq3^r-1$.
Thus $n\geq T^*(P)$ implies
\eqref{eq:parameter-density-criterion}.

For the second assertion, let
$A_0=\prod_i(p_i-2)$.  Since $A_0\geq3^r$,
\[
T^*(P)
\leq
3P-\frac{P}{3^{r-1}}+2.
\]
Also $P\geq5^r\geq2\cdot3^{r-1}$, so the right-hand side is at
most $3P$.
\end{proof}

\begin{corollary}[Product cutoff]
\label{cor:product-cutoff}
If
\[
P\leq\sqrt{2n},
\]
then $g_P(2n)>0$.
\end{corollary}

\begin{proof}
If $P\geq7$, then
$n\geq P^2/2\geq3P\geq T^*(P)$.
For $P=5$, the hypothesis gives $n\geq13$, while
$T^*(5)=12$.  Corollary~\ref{cor:residue-free-threshold}
applies in both cases.
\end{proof}

\subsection{A threshold depending only on the largest allowed prime}

Let $p\geq5$ be prime and write
\begin{equation}
\label{eq:r-p-definition}
r_p=\pi(p)-2,
\qquad
P_{\leq p}
=
\prod_{\substack{q\ \mathrm{prime}\\5\leq q\leq p}}q.
\end{equation}

\begin{theorem}[Largest-prime uniform positivity threshold]
\label{thm:largest-prime-threshold}
Define
\begin{equation}
\label{eq:N-uniform-definition}
\mathcal N(p)
=
6\bigl(3^{r_p}-1\bigr)
\prod_{\substack{q\ \mathrm{prime}\\5\leq q\leq p}}
\frac{q}{q-2}
+4.
\end{equation}
If $P'>1$ is any squarefree product of primes in $[5,p]$, then
\begin{equation}
\label{eq:largest-prime-uniform-positive}
2n\geq\mathcal N(p)
\quad\Longrightarrow\quad
g_{P'}(2n)>0.
\end{equation}
\end{theorem}

\begin{proof}
Write $P'=q_1\cdots q_s$.  Since $s\leq r_p$ and every factor
$q/(q-2)$ is larger than $1$,
\[
3(3^s-1)
\prod_{j=1}^{s}\frac{q_j}{q_j-2}
+2
\leq
\frac{\mathcal N(p)}2.
\]
Thus $2n\geq\mathcal N(p)$ implies
$n\geq T^*(P')$, and Corollary~\ref{cor:residue-free-threshold}
applies.
\end{proof}

Let
\begin{equation}
\label{eq:twin-prime-constant}
C_2
=
\prod_{\substack{q\ \mathrm{prime}\\q\geq3}}
\left(1-\frac{1}{(q-1)^2}\right)
\end{equation}
be the twin-prime constant, and let $\gamma$ be Euler's constant.

\begin{proposition}[Asymptotic size of the uniform threshold]
\label{prop:uniform-threshold-asymptotic}
As $p\to\infty$ through the primes,
\begin{equation}
\label{eq:N-uniform-asymptotic}
\mathcal N(p)
\sim
\frac{e^{2\gamma}}{2C_2}\,
3^{\pi(p)-2}(\log p)^2,
\end{equation}
and therefore
\begin{equation}
\label{eq:log-N-uniform-asymptotic}
\log\mathcal N(p)
\sim
(\log3)\frac{p}{\log p}.
\end{equation}
\end{proposition}

\begin{proof}
For every prime $q\geq5$,
\[
1-\frac2q
=
\left(1-\frac1q\right)^2
\left(1-\frac{1}{(q-1)^2}\right).
\]
Mertens' product theorem \cite{HardyWright2008} and the
definition of $C_2$ give
\[
\prod_{\substack{q\ \mathrm{prime}\\5\leq q\leq p}}
\left(1-\frac2q\right)
\sim
\frac{12C_2e^{-2\gamma}}{(\log p)^2}.
\]
Taking reciprocals in this relation and substituting into
\eqref{eq:N-uniform-definition} proves
\eqref{eq:N-uniform-asymptotic}.  The prime number theorem,
$\pi(p)\sim p/\log p$, then gives
\eqref{eq:log-N-uniform-asymptotic}.
\end{proof}

\begin{remark}[Comparison of three scales]
\label{rem:three-positivity-scales}
The half-period threshold, uniformly over squarefree products
supported on the primes in $[5,p]$, is
\[
\mathcal H(p)
=
6P_{\leq p}
=
\prod_{q\leq p}q.
\]
Since $T^*(P_{\leq p})\leq3P_{\leq p}$,
\[
\mathcal N(p)\leq\mathcal H(p)
\qquad(p\geq5).
\]
The prime number theorem gives
\[
\log\mathcal H(p)\sim p,
\qquad
\log\mathcal N(p)
\sim
(\log3)\frac{p}{\log p}.
\]
Thus the explicit density threshold is subexponential in $p$,
whereas the half-period threshold has primorial scale.

The natural condition $p\leq\sqrt{2n}$ corresponds to the
quadratic scale $2n\geq p^2$.  Since
$\mathcal N(p)/p^2\to\infty$, the three scales satisfy
\[
p^2=o(\mathcal N(p)),\qquad \mathcal N(p)=o(\mathcal H(p))
\]
as $p\to\infty$.  This comparison records the ranges delivered
by the elementary finite-density argument and the half-period
argument.
\end{remark}

\begin{remark}[Comparison with lower-bound sieves]
\label{rem:sieve-comparison}
The parameter problem fits the framework of a sieve of dimension
at most $2$. This yields polynomial sufficient scales in the largest
allowed prime, in contrast to the elementary threshold
$\mathcal N(p)$. We record the reason for this comparison.

Put $X=m(n)+1$. For a squarefree divisor $d$ of $P$, the excluded
parameter classes have cardinality
\[
\rho_n(d)=\prod_{q\mid d}(2-\mathbf1_{\{\delta_q(n)=0\}}).
\]
The corresponding interval count is $X\rho_n(d)/d+R_d$, where
$|R_d|\leq\rho_n(d)\leq2^{\omega(d)}$, uniformly in $n$.
Here $\omega(d)$ is the number of distinct prime factors of $d$.
For a uniform lower bound, one may enlarge each excluded set to
exactly two classes at every prime in $[5,p]$, including primes
absent from the original support. Survivors of this more restrictive
sieve are admissible in the original problem. Its local density
is $2/q$, and its remainder is bounded by $2^{\omega(d)}$.
For each fixed integer $K\geq1$, the divisor bound
\[
\sum_{d\leq D}K^{\omega(d)}\ll_K D(1+\log D)^{K-1}
\]
(for integer $K$, by comparison with the $K$-fold divisor function)
controls the weighted remainder sums required by the sieve.
Thus one can use $D=X/(\log X)^B$ with a sufficiently large fixed
$B$.

The DHR lower-bound sieve \cite{DiamondHalberstam2008} has a positive
lower sieve function when $\log D/\log p>\beta_2$, where
$\beta_2\approx4.266$; see the comparison table in \cite{Franze2010}.
Consequently, for each fixed $c>\beta_2$ and all sufficiently large
$p$ (depending on $c$), it gives positivity whenever $2n\geq p^c$,
uniformly over the allowed squarefree products $P'>1$.
Indeed $X\sim n/3$, so this range permits a sieve level with
$\log D/\log p$ bounded strictly above $\beta_2$ and a negligible
weighted remainder. If $X$ is larger, the same argument can be
applied to an initial subinterval of suitable length.

The elementary bound $\mathcal N(p)$ is therefore not an
asymptotically optimal sieve threshold. Its role here is to give
an unconditional, fully explicit inequality for every permitted
finite modulus without additional asymptotic parameters. For each
fixed $c>\beta_2$, the scales satisfy
\[
p^2=o(p^c),\qquad p^c=o(\mathcal N(p)),\qquad
\mathcal N(p)=o(\mathcal H(p)).
\]
\end{remark}

\begin{example}[Size of the explicit sufficient threshold]
\label{ex:threshold-table}
The following finite computations compare $T^*(P_{\leq p})=
\mathcal N(p)/2$ with the actual last zero for several small supports.
\begin{center}
\begin{tabular}{r|r|r|r}
$p$ & $\mathcal N(p)/2$ & last $n$ with $g_{P_{\leq p}}(2n)=0$ & $p^2/2$\\
\hline
5 & 12.0 & 8 & 12.5\\
7 & 58.0 & 8 & 24.5\\
11 & 224.4 & 14 & 60.5\\
13 & 810.9 & 14 & 84.5\\
17 & 2775.1 & 20 & 144.5\\
19 & 9325.8 & 28 & 180.5\\
23 & 30665.4 & 32 & 264.5
\end{tabular}
\end{center}
The thresholds in the second column are rounded to one decimal
place. The last zeros were obtained by direct integer coprimality
counts for $1\leq n<\lceil T^*(P_{\leq p})\rceil$; the sufficient
threshold excludes every later zero. Monotonicity in the prime
support implies that the last zero for the complete product is
also the maximum last zero over all its nonempty subproducts.
The estimates need not be close to the actual transition:
for example, over $1\leq n\leq6P$, the maximum of $|\Bcal_P(n)|$
is $12/7$ for $P=35$ and $324/77$ for $P=385$, compared with the
bounds $3^2-1=8$ and $3^3-1=26$.
\end{example}

\section{The complete prime cutoff}

For every $n\geq13$, define
\begin{equation}
\label{eq:Pi-definition}
\Pi(n)
=
\prod_{\substack{p\ \mathrm{prime}\\
5\leq p\leq\sqrt{2n}}}p.
\end{equation}
The assumption $n\geq13$ is equivalent, for integer $n$, to
$5\leq\sqrt{2n}$, and ensures that the product is nonempty.
Thus $\Pi(n)$ contains every prime, other than $2$ and $3$, up
to the natural factorization threshold $\sqrt{2n}$.

The notation $g_{\Pi(n)}(2n)$ is understood pointwise. For each
fixed $n$, the preceding theory is applied with
\[
P=\Pi(n).
\]
Since this modulus varies with $n$, the affine periodicity
theorem for a fixed $P$ is not being applied to the sequence
$g_{\Pi(n)}(2n)$.

\begin{proposition}[Monotonicity and extremality]
\label{prop:cutoff-monotonicity}
Let $P$ and $Q$ be squarefree products of primes at least $5$.
If $P\mid Q$, then, for every positive integer $n$,
\begin{equation}
\label{eq:monotonicity}
g_Q(2n)\leq g_P(2n).
\end{equation}
Consequently, if every prime factor of $P$ is at most
$\sqrt{2n}$, then
\begin{equation}
\label{eq:cutoff-extremal}
g_{\Pi(n)}(2n)\leq g_P(2n).
\end{equation}
In particular, $g_{\Pi(n)}(2n)>0$ if and only if
$g_P(2n)>0$ for every squarefree product $P$ of primes in the
interval $[5,\sqrt{2n}]$.
\end{proposition}

\begin{proof}
Every representation whose summands are coprime to $6Q$ also
has both summands coprime to $6P$. Hence the set counted by
$g_Q(2n)$ is contained in the set counted by $g_P(2n)$, which
proves \eqref{eq:monotonicity}. Under the stated cutoff on the
prime factors, one has $P\mid\Pi(n)$, and
\eqref{eq:cutoff-extremal} follows.

For the final assertion, positivity at the complete cutoff
implies positivity for every such $P$ by
\eqref{eq:cutoff-extremal}. The converse follows by taking
$P=\Pi(n)$.
\end{proof}

\begin{proposition}[Prime detection at the cutoff]
\label{prop:cutoff-prime}
Let $n\geq13$ and let
\[
2\leq u\leq2n-1.
\]
If
\[
\gcd\bigl(u,6\Pi(n)\bigr)=1,
\]
then $u$ is a prime and
\[
u>\sqrt{2n}.
\]
\end{proposition}

\begin{proof}
Suppose that $u$ is composite. Then $u$ has a prime divisor $q$
satisfying
\[
q\leq\sqrt u<\sqrt{2n}.
\]
If $q\in\{2,3\}$, then
\[
\gcd\bigl(u,6\Pi(n)\bigr)>1.
\]
If $q\geq5$, then $q$ occurs among the prime factors of
$\Pi(n)$, which gives the same contradiction. Therefore $u$
is prime.

It remains to locate this prime. If $u\leq\sqrt{2n}$, then
$u$ itself occurs among the prime factors of $6\Pi(n)$. This
again contradicts the coprimality hypothesis. Hence
$u>\sqrt{2n}$.
\end{proof}

\begin{corollary}[Exact cutoff count]
\label{cor:cutoff-equivalence}
For $n\geq13$, define
\[
G_{>\sqrt{2n}}(2n)=
\#\left\{(q_1,q_2):
\begin{array}{l}
q_1,q_2\text{ are primes},\\
\sqrt{2n}<q_1\leq q_2,\quad q_1+q_2=2n
\end{array}\right\}.
\]
Then
\begin{equation}
\label{eq:cutoff-exact-count}
g_{\Pi(n)}(2n)=\mathbf1_{\{2n-1\text{ is prime}\}}
+G_{>\sqrt{2n}}(2n).
\end{equation}
In particular,
\begin{equation}
\label{eq:cutoff-equivalence}
g_{\Pi(n)}(2n)>0
\iff
\bigl(2n-1\text{ is prime}\bigr)
\ \text{or}\ G_{>\sqrt{2n}}(2n)>0.
\end{equation}
\end{corollary}
\begin{proof}
An admissible representation with smaller summand $1$ must be
$1+(2n-1)$. By Proposition~\ref{prop:cutoff-prime}, it is admissible
exactly when $2n-1$ is prime: in that case $2n-1>\sqrt{2n}$.
Every other admissible representation has both summands at least
$2$. The same proposition shows that both are primes exceeding
$\sqrt{2n}$. Conversely, any such pair consists of two integers
coprime to $6\Pi(n)$. These two disjoint classes of representations
give \eqref{eq:cutoff-exact-count}, and the positivity equivalence
follows.
\end{proof}

\begin{remark}[Additive-prime interpretation]
\label{rem:additive-prime-cutoff}
Corollary~\ref{cor:cutoff-equivalence} separates the endpoint
$1+(2n-1)$ from the interior of the interval.  When $2n-1$ is
composite, positivity of $g_{\Pi(n)}(2n)$ is equivalent to the
existence of primes $q_1,q_2>\sqrt{2n}$ with
$q_1+q_2=2n$.  When $2n-1$ is prime, the endpoint
$1+(2n-1)$ contributes an additional admissible representation.
\end{remark}

\begin{remark}[Size of the cutoff product]
\label{rem:cutoff-product-size}
Let
\[
\vartheta(x)
=
\sum_{\substack{q\leq x\\q\ \mathrm{prime}}}\log q.
\]
Then
\[
\log\Pi(n)
=
\vartheta\bigl(\sqrt{2n}\bigr)-\log6,
\]
so the prime number theorem gives
\[
\Pi(n)
=
\exp\left((1+o(1))\sqrt{2n}\right).
\]
The product cutoff $P\leq\sqrt{2n}$ does not apply to
$P=\Pi(n)$ once the prime $7$ enters the product, that is, from
$n=25$ onward.  Indeed, let $q^+$ denote the next prime after a prime $q$. If
$q\leq\sqrt{2n}<q^+$ with $q\geq7$, Bertrand's postulate \cite{HardyWright2008}
gives $\sqrt{2n}<q^+<2q$, whereas
$\Pi(n)\geq5q>\sqrt{2n}$.

The distinction between a bound on the product $P$ and a bound
on its largest prime is thus substantial.  The uniform threshold
$\mathcal N(p)$ of Theorem~\ref{thm:largest-prime-threshold}
provides an explicit intermediate scale, as summarized in
Remark~\ref{rem:three-positivity-scales}.
\end{remark}

\begin{remark}[Scope of positivity at the complete cutoff]
\label{rem:scope-complete-cutoff}
The residue-free threshold and the residue-sensitive criterion
have different scopes. For example, $\mathcal N(5)=24$ guarantees
positivity at the complete cutoff for $13\leq n\leq24$, while
$\mathcal N(7)=116$ gives $58\leq n\leq60$. In the next prime
band, $61\leq n\leq84$, the bound $\mathcal N(11)=4040/9$
exceeds every available value of $2n$.
Nevertheless, Theorem~\ref{thm:parameter-density-positive} can still
apply outside those ranges. At $n=30$, $P=\Pi(n)=35$, its two sides
are $40/7$ and $5$. At $n=77$, $P=\Pi(n)=385$, they are $900/77$
and $11$. Thus the residue-sensitive criterion proves positivity
in both cases.

Asymptotically, put $x=2n$ and let $p$ be the greatest prime not
exceeding $\sqrt{x}$. The prime number theorem implies
$p\sim\sqrt{x}$, and hence
\[
\log\mathcal N(p)\sim
2(\log3)\frac{\sqrt{x}}{\log x},
\qquad
\frac{\log\mathcal N(p)}{\log x}\longrightarrow\infty.
\]
Thus $x\geq\mathcal N(p)$ fails for all sufficiently large $x$
at this cutoff. The polynomial sieve scales in
Remark~\ref{rem:sieve-comparison} also do not reach $x\asymp p^2$.
The results of this paper do not establish
$g_{\Pi(n)}(2n)>0$ for every $n\geq13$.
\end{remark}

\section{Conclusion}

For every squarefree
\[
P=p_1\cdots p_r,
\qquad
5\leq p_1<\cdots<p_r,
\]
the restricted representation count $g_P(2n)$ admits an explicit
finite CRT formula obtained from the local excluded residue
classes.  The corresponding complete-modulus correlation is the
Nagell totient $\theta(2n,6P)$ and has cardinality
$\kappa_P(n)$.  This quantity is also the affine increment of
$g_P(2n)$ and the finite local-density factor.

The parameter-density decomposition yields the residue-sensitive
criterion \eqref{eq:parameter-density-criterion} and, uniformly in
the residue profile,
\[
n\geq
3(3^r-1)
\prod_{i=1}^{r}\frac{p_i}{p_i-2}+2
\quad\Longrightarrow\quad
g_P(2n)>0.
\]
This threshold is at most $3P$. The half-period positivity bound
also follows directly from reflection and complete-block counting;
the product cutoff follows from these sufficient thresholds.  Uniformly over
all squarefree products supported on primes at most $p$, the
corresponding even threshold $\mathcal N(p)$ satisfies
\[
\log\mathcal N(p)
\sim
(\log3)\frac{p}{\log p},
\]
which places it between the polynomial sieve scales and the
primorial half-period scale. Its explicit elementary form is useful
for finite parameters, although it is not asymptotically competitive
with lower-bound sieves.

At the complete cutoff
\[
\Pi(n)=
\prod_{\substack{5\leq q\leq\sqrt{2n}\\q\ \mathrm{prime}}}q,
\]
the surviving interior summands are precisely primes exceeding
$\sqrt{2n}$.  This gives the exact additive-prime
characterization in Corollary~\ref{cor:cutoff-equivalence} and
connects the finite-modulus formulas with the natural
square-root cutoff. This characterization does not establish
universal positivity at that cutoff.


\end{document}